\documentclass{amsart}
\usepackage{amssymb, amsmath, tikz-cd}
\usepackage{stmaryrd}
\usepackage{xcolor}

\title[Solutions to certain differential and difference equations]{Effective finiteness of solutions to certain differential and difference equations}
\author{Patrick Ingram}
\date{\today}
\address{York University, Toronto, Canada}
\thanks{The author would like to thank Gary Gundersen and two anonymous referees for helpful comments on an earlier draft of this note. This research was supported by a grant from NSERC}

\renewcommand{\epsilon}{\varepsilon}
\renewcommand{\phi}{\varphi}

\newcommand{\PP}{\mathbb{P}}

\newcommand{\CC}{\mathbb{C}}

\newcommand{\QQ}{\mathbb{Q}}

\newcommand{\Res}{\operatorname{Res}}

\newtheorem{lemma}{Lemma}

\newtheorem{theorem}[lemma]{Theorem}
\newtheorem{corollary}[lemma]{Corollary}

\newtheorem{question}[lemma]{Question}

\theoremstyle{definition}
\newtheorem{example}[lemma]{Example}
\newtheorem{remark}[lemma]{Remark}

\begin{document}
\begin{abstract} 
For $R(z, w)\in \CC(z, w)$ of degree at least 2 in $w$, we show that the number of rational functions $f(z)\in \CC(z)$ solving the difference equation $f(z+1)=R(z, f(z))$ is finite, bounded just in terms of the degrees of $R$ in the two variables. This complements a result of Yanagihara, who showed that any finite-order meromorphic solution to this sort of difference equation must be a rational function. We prove a similar result for the differential equation $f'(z)=R(z, f(z))$, building  on a result of Eremenko.
\end{abstract}
\maketitle

Malmquist~\cite{MR1555091} showed that if $R(z, w)\in \CC(z, w)$ is a rational function, and $f(z)$ is a meromorphic solution to the differential equation
\begin{equation}\label{eq:diff}f'(z)=R(z, f(z)),\end{equation} 
 then either $R(z, w)$ is a polynomial of degree at most 2 in $w$ (and hence~\eqref{eq:diff} is a linear or Ricatti equation), or else $f(z)$ is a rational function.  Eremenko~\cite{MR1601867} established a bound on the degree of $f$ in the latter case (for more general first-order ODEs). In the case that $R$ is a polynomial in both variables, Gundersen~\cite{MR2419455} established bounds on the number of solutions to~\eqref{eq:diff}, in terms of the degree of $R$ in $w$ and the number of distinct roots $z$ of the leading coefficient (see also \cite{MR2232199, MR808678}).
 
A difference-equation analogue of Malmquist's Theorem was derived by Yanagihara~\cite{MR621536}, who showed that any finite-order meromorphic solution $f$ to
\begin{equation}\label{eq:eq}f(z+1)=R(z, f(z)),\end{equation}
is rational, assuming that $\deg_w(R)\geq 2$. 
The purpose of this note is to establish a  result complementary to Yanagihara's, specifically that the number of rational solutions to~\eqref{eq:eq} is finite and bounded just in terms of the degree of $R$ in each variable. Indeed, our proof is effective in the sense that it gives us an in-principle computable list of rational functions which must contain all solutions. As our methods apply to~\eqref{eq:diff} with minor modifications, and offer a different approach to computing the finite set of solutions in certain cases, we treat that as well, although in the differential context this largely amounts to a new approach to a known result.

\begin{theorem}\label{th:finite}
Let $R(z, w)\in \CC(z, w)$. There exist explicit constants $B_1$ and $B_2$, depending just on $\deg_w(R)$ and $\deg_z(R)$, such that the following hold:
\begin{enumerate}
\item  If $\deg_w(R)\geq 2$, then there are at most $B_1$  rational functions $f$ solving~\eqref{eq:eq}, and the set of solutions is effectively computable. 
\item If $\deg_w(R)\geq 3$, then there are at most $B_2$ rational functions $f$ solving~\eqref{eq:diff}, and the set of solutions is effectively computable.\end{enumerate}
\end{theorem}

Indeed, we show that we may take
\begin{equation}\label{eq:B}B_x=(\deg_w(R)+x)^{\deg_z(R)}\frac{(\deg_w(R)+x)^{(\deg_w(R)+x)(3\deg_z(R)+1)}-1}{(\deg_w(R)+x)^{(\deg_w(R)+x)}-1},\end{equation}
although these are surely not optimal bounds, and refinements here would be of interest.

Our proof proceeds roughly as follows, focusing on the difference equation case. First, we bound the degree of a rational function $f$ solving~\eqref{eq:eq}. This is similar in flavour to the argument behind Yanagihara's result, which makes estimates on the Nevanlinna characteristic of a putative solution, but once restricted to the setting of rational functions we can do this by purely elementary means. We then show that the rational solutions to~\eqref{eq:eq} of a given degree correspond to an algebraic subset of some projective space, whose irreducible components have degree summing to at most some bound which depends only on the degrees of $R$ in the two variables. This would prove the result, but for the possibility that some of these irreducible components have positive dimension.

Changing gears, we use an arithmetic argument (a height bound) to show that, if $R$ and the solutions to ~\eqref{eq:eq} in question happen to all have algebraic coefficients, then the Zariski closure of this set of solutions cannot have any components of positive dimension. This is then the base case of an induction on transcendence rank, which proves the result over any finitely-generated subfield of $\CC$. Since $R$ and any finite collection of solutions to~\eqref{eq:eq} can be defined over some finitely-generated extension, this resolves the general case. The crux of the induction is essentially that we may replace some transcendental values appearing in the coefficients of our various rational functions with values from some subfield, in such a way as to preserve~\eqref{eq:eq}. The consideration of the heights of the coefficients of solutions seems to be novel here, although Eremenko's arguments in~\cite{MR1601867} use the related function-field height.

While our results produce an effectively computable finite set, the computations involved are not necessarily practical, even in simple cases. For instance,  consider the following variation of an example of Yanagihara:
\[f(z+1)=f(z)+1+\frac{2z^3}{f(z)}.\]
 Then it is a consequence of the various lemmas below than any rational solution  has degree at most 9, and any solution in $\QQ(z)$ can be written with integer coefficients of absolute value at most $8.2\times 10^{40}$, which gives a finite search space, but one too large to exhaust. (And, \emph{a priori}, these might not be all solutions, as our proof provides here only that the coefficients will be algebraic of degree at most $8\times 10^{12}$.)
Meanwhile, for this example an elementary consideration of the zeros and poles of a putative solution show that we can have only $f(z)=z^2$. 

We also note that while we prove bounds on the size of the set of solutions below when $R$ has coefficients in $\CC$, we only discuss computation of the finite set of solutions (which depends on the theory of heights) when $R$ has coefficients in $\overline{\QQ}$. This is only for simplicity, though, and if $R$ has transcendental coefficients, one may simply appeal to the (somewhat more complicated) theory of heights over finitely generated extensions of $\QQ$ due to Moriwaki~\cite{MR1779799}.

This paper raises a few questions, and we mention two here for future consideration. First, an anonymous referee proposes the following.
\begin{question}
In analogy with the statement of Malmquist's Theorem, is it true that the number of solutions to~\eqref{eq:diff} is bounded as in Theorem~\ref{th:finite} whenever $R$ is not linear or a quadratic polynomial in $w$?
\end{question}
The second question is motivated by Remark~\ref{rm:gendiff} below.
\begin{question}
Given an factional linear transformation $\sigma(z)\in\CC(z)$, do there exist any finite-order meromorphic solutions to $f\circ \sigma(z)=R(z, f(z))$ other than rational solutions?
\end{question}
Note that if one restricts $\sigma$ to be a affine linear transformation, then one can make some progress by combining Yanagihara's proof from~\cite{MR621536} with an estimate of Bergweiler~\cite{MR1257098} on the characteristic of a composition of functions, but even this appears to give a much weaker result.

\section{Degrees of solutions}

Our first lemma is a standard result on the elimination of variables. For the rest of the paper we set $d=\deg_w(R)$. In general, if $F$ is a polynomial in several variables, $\deg(F)$ will mean the \emph{total degree} of $F$.
\begin{lemma}\label{lem:res}
Given $P, Q\in \CC[z, X_1, X_0]$, homogeneous of degree $d$ in $X_1$ and $X_0$ and with no common factor, there exists a non-zero $\Res(P, Q)\in\CC[z]$ and  $A_i, B_i\in \CC[z, X_1, X_0]$, homogeneous  of degree $d-1$ in the $X_i$, with
\begin{equation}\label{eq:resultant}\Res(P, Q)X_i^{2d-1}=A_i(X_1, X_0)P(X_1, X_0)+B_i(X_1, X_0)Q(X_1, X_0)\end{equation}
for $i=0, 1$. 	Furthermore, $\Res(P, Q)$ is the determinant of some $2d\times 2d$ matrix with entries which are coefficients of $P$ and $Q$, and every coefficient of $A_i$ and $B_i$ is the determinant of some $(2d-1)\times (2d-1)$ matrix, with entries again coefficients of $P$ and $Q$.
\end{lemma}

\begin{proof}
This is a standard result, which we sketch here for completeness (see~\cite[Lemma~1]{defects} for more details). Writing $H_d$ for the $\CC[z]$-module of homogeneous forms of degree $d$ in $X_1$ and $X_0$, note that
\begin{equation}\label{eq:linmap}(A, B)\mapsto AP+BQ\end{equation}
is a linear map from $H_{d-1}\times H_{d-1}$ to $H_{2d-1}$. We define $\Res(P, Q)$ to be the determinant of the Sylvester matrix of $P$ and $Q$, which is the the coordinate matrix of~\eqref{eq:linmap} (relative to the natural bases). If $\Res(P, Q)\neq 0$, then it follows from Cramer's Rule that we can solve~\eqref{eq:resultant}, and that the coefficients of the solutions will be determinants of the Sylvester matrix with certain columns replaced by the standard basis vector representing $X_i^{2d-1}$. In the case that $\Res(P, Q)=0$, the kernel of~\eqref{eq:linmap} will contain a nontrivial element, which corresponds to a pair $(A, B)$ such that $A/B$ is a rational function of degree at most $d-1$, and is equal to $-P/Q$. This means that $P$ and $Q$ have a common factor.
\end{proof}

We now give a lower bound on the degree of a certain composition of rational functions. 
For the rest of the paper, we represent $R$ as
\[R(z, X/Y)=\frac{P(X, Y)}{Q(X, Y)}\]
for homogeneous forms $P, Q$ of degree $d$ with no common factor, and coefficients in $\CC[z]$.

We begin with an explicit estimate on the degree of a rational function of a rational function with polynomial coefficients.
\begin{lemma}\label{lem:degree} With $d\geq 2$, we have
\[d\deg(f)\leq \deg(R(z, f(z)))+(2d-1)\deg_z(R).\]
\end{lemma}

 Note that an estimate of this form already follows from a result of Monhon'ko~\cite{MR0298006}, without an explicit error term. Writing $T_r(f)$ for the Nevanlinna charateristic function of $f$, then $T_r(f)=\deg(f)\log r+O(1)$ as $r\to\infty$. The main result of\cite{MR0298006}, is that
 \[dT_r(f)=T_r(R(z, f(z)))+O\left(\sum T_r(c_i)\right)\]
  where the $c_i$ are the coefficients of $R$, from which we deduce that $d\deg(f)=\deg(R(z, f(z)))+O(\deg_z(R))$, where the implied constant depends on $d$.

\begin{proof}[Proof of Lemma~\ref{lem:degree}]
Set $f=f_1/f_0$, where $f_1$ and $f_0$ are polynomials with no common factor, and write
\begin{equation}
\label{eq:res}	
\Res(P, Q)f_i^{2d-1}=A_i(f_1, f_0)P(f_1, f_0)+B_i(f_1, f_0)Q(f_1, f_0)
\end{equation}
as in Lemma~\ref{lem:res}.
Now, each $A_i(X, Y)$ has degree $d-1$ and coefficients which are determinants of $(2d-1)\times (2d-1)$ matrices, whose entries are coefficients of $P$ and $Q$. It follows that
\begin{align*}
\deg(A_i(f_1, f_0))&\leq (d-1)\max\{\deg(f_1), \deg(f_0)\} + (2d-1)\deg_z(R)\\
&=(d-1)\deg(f)+(2d-1)\deg_z(R),
\end{align*}
and similarly for $B_i$. So we have from~\eqref{eq:res} that
\begin{align*}
\deg(\Res(P, Q))+(2d-1)\deg(f)&= \max_{i=0, 1}\deg(\Res(P, Q)f_i^{2d-1})\\
&\leq \max_{i=0, 1}\deg\big(A_i(f_1, f_0)P(f_1, f_0)\\&\quad +B_i(f_1, f_0)Q(f_1, f_0)\big)\\
&\leq   \max\{\deg(P(f_1, f_0)), \deg(Q(f_1, f_0))\}\\&\quad+(d-1)\deg(f)+(2d-1)\deg_z(R).
\end{align*}

On the other hand, any common factor in $\CC[z]$ of $P(f_1, f_0)$ and $Q(f_1, f_0)$ must divide $\Res(P, Q)$, and so rearranging the above gives
\begin{align*} d\deg(f) &\leq  \max\{\deg(P(f_1, f_0)), \deg(Q(f_1, f_0))\}-\deg(\Res(P, Q))\\&\quad+(2d-1)\deg_z(R)\\
	&\leq  \deg\left(\frac{P(f_1, f_0)}{Q(f_1, f_0)}\right)+(2d-1)\deg_z(R)\\
	&=\deg\big(R(z, f(z))\big)+(2d-1)\deg_z(R).
\end{align*}
\end{proof}

Note that it is not hard to construct examples in which we see that some error term in Lemma~\ref{lem:degree} is necessary, although it is not clear how sharp the estimate is. For instance, if $R(z, w)=w^d-z^{dm}$, then $R(z, z^m)=0$, showing that we cannot replace the factor of $2d-1$ in the error term by anything less than 1.

Our next lemma restricts the degree of a solution to~\eqref{eq:eq}, making the form of a hypothetical solution more concrete. In the difference-equation context, this lemma can be obtained from Yanagihara's argument, by making the error term in a result of Valiron~\cite{MR1504970} more explicit, but the previous lemma is already enough. For the context of~\eqref{eq:diff}, the argument is similar to that of Eremenko~\cite{MR1601867}, with the previous lemma doing most of the work.

\begin{lemma}\begin{enumerate}
\item If $f$ satisfies~\eqref{eq:eq} and $d\geq 2$, then $\deg(f)\leq 3\deg_z(R)$. 
\item If $f$ satisfies~\eqref{eq:diff} and $d\geq 3$, then $\deg(f)\leq 5\deg_z(R)$	
\end{enumerate}
\end{lemma}

\begin{proof}
For the first claim note that it follows from Lemma~\ref{lem:degree} that for $f$ satisfying~\eqref{eq:eq},
\begin{align*}
d\deg(f)&\leq \deg\big(R(z, f(z))\big)	+(2d-1)\deg_z(R)\\
&=\deg(f(z+1))	+(2d-1)\deg_z(R)\\
&=\deg(f)+(2d-1)\deg_z(R),
\end{align*}
whence
\[\deg(f)\leq \frac{2d-1}{d-1}\deg_z(R)\leq 3\deg_z(R).\]
On the other hand, $\deg(f')\leq 2\deg(f)-2$, by the quotient rule, and so from Lemma~\ref{lem:degree} solutions to~\eqref{eq:diff} have
\[d\deg(f)\leq 2\deg(f)-2+(2d-1)\deg_z(R)\leq 2\deg(f)+(2d-1)\deg_z(R),\]
whence
\[\deg(f)\leq \frac{2d-1}{d-2}\deg_z(R)\leq 5\deg_z(R)\]
for $d\geq 3$.
\end{proof}

\section{Solutions of a given degree}

We now focus on solutions to~\eqref{eq:eq} of fixed degree. A rational function of degree $k$
\begin{equation}\label{eq:ratd}f(z)=\frac{c_0+\cdots +c_kz^k}{c_{k+1}+\cdots +c_{2k+1}z^k}\end{equation}
can be identified with the point in projective space $\mathbf{c}=[c_0:\cdots :c_{2k+1}]\in \PP^{2k+1}$, but not every point in $\PP^{2k+1}$ gives a rational function of the right degree. In particular, the resultant of the numerator and denominator in~\eqref{eq:ratd} is a homogeneous form $\Res(\mathbf{c})$ in the coordinates of $\PP^{2k+1}$ of degree $2k$, and rational functions of degree exactly $k$ correspond to points on $\operatorname{Hom}_k\subseteq \PP^{2k+1}$, the complement of the  hypersurface defined by $\Res(\mathbf{c})=0$.

Our next lemma requires the machinery of heights. Let $K$ be a number field, and let $|\cdot|_v$ be an absolute value on $K$ whose restriction to $\QQ$ is either the usual absolute value, or a $p$-adic absolute value. The set of such $v$ will be denoted by $M_K$. For a point $P=[P_0:\cdots :P_N]\in\PP^N(K)$, we define the \emph{logarithmic Weil height} $h(P)$ by
\begin{equation}\label{eq:height}h(P)=\sum_{v\in M_K}\frac{[K_v:\QQ_v]}{[K:\QQ]}\log\max\{|P_0|_v, ..., |P_N|_v\},\end{equation}
where $K_v$ is the completion of $K$ with respect to $|\cdot|_v$.
It is a standard result (see, e.g., \cite[p.~176]{MR1745599}) that $h(P)$ is independent both of the choice of homogeneous coordinates representing $P$, and of the field $K$. That is, $h$ is a non-negative, well-defined function on $\PP^{N}(\overline{K})$. We will write $h(f)$ for the height of the tuple of coefficients of $f(z)\in \overline{\QQ}(z)$, when $f$ is written as in~\eqref{eq:ratd}, as a quotient of polynomials with no common factor.

The main utility of heights, for our purposes, will be the following finiteness result (see, e.g., \cite[Theorem~B.2.3, p.~177]{MR1745599}).
\begin{lemma}[Northcott~\cite{MR34607}]\label{lem:northcott}
	For any finite $B$ and $D$, the set of points $P\in \PP^N(\overline{\QQ})$ with $h(P)\leq B$ defined over number fields of degree at most $D$ is finite and effectively computable.
\end{lemma}

Northcott's Theorem need not be particularly mysterious, at least in the case $D=1$. If $P=[P_0:\cdots :P_N]\in \PP^N(\QQ)$, then by scaling the coordinates we may take  the $P_i$ to be integers not sharing a common factor. In this case, one checks from~\eqref{eq:height} that
\[h(P)=\log\max\{|P_0|, ..., |P_N|\},\]
and of course bounding this allows only finitely many choices for the $P_i$.

We will use heights to show finiteness in certain cases of the proof of the main result, via the next lemma, which is an arithmetic analogue of Lemma~\ref{lem:degree}.

\begin{lemma}\label{lem:height} 
Let $R(z, w)\in \overline{\QQ}(z, w)$ with $d=\deg_w(R)\geq 2$, and let $f(z)\in \overline{\QQ}(z)$. Then
\[dh(f)\leq h(R(z, f(z)))+O(1),\]
where the implied constant depends only on $R$ and $\deg(f)$.
\end{lemma}

\begin{proof}
Let $K$ be some number field containing the coefficients of  $R$ and $f$, and let $|\cdot|_v$ be an absolute value on $K$. For a polynomial $F(x_1, ..., x_n)\in K[x_1, .., x_n]$ in however many variables, we set
\[\|F\|_v=\left\|\sum a_{i_1, ..., i_n}x_1^{i_1}\cdots x_n^{i_n}\right\|_v=\max|a_{i_1, ..., i_n}|_v.\]
We will also set
\[\|F_1, ..., F_m\|_v=\max\{\|F_1\|_v, ..., \|F_m\|_v\}.\]

If $|\cdot|_v$ is a non-archimedean absolute value, then 
\[\| F_1+\cdots+F_m\|_v\leq \| F_1,  ..., F_m\|_v\]
by the strong triangle inequality, while
\[\left\|\prod_{i=1}^mF_i\right\|_v=\prod_{i=1}^m\| F_i\|_v\]
by the Gau\ss\ Lemma~\cite[Lemma~1.6.3, p.~22]{MR2216774}.

In the case of an archimedean absolute value, the triangle inequality gives
\[\| F_1+\cdots+F_m\|_v\leq m\| F_1,  ..., F_m\|_v.\]
Somewhat less obviously, in this case we have, for polynomials in $n$ variables,
\[\prod_{i=1}^m \|F_i\|_v2^{-n\deg(F_i)}\leq \left\|\prod_{i=1}^m F_i\right\|_v \leq \prod_{i=1}^m \|F_i\|_v2^{n\deg(F_i)} \]
This is due to Mahler~\cite{MR138593}, or by applying Gelfond's Lemma~\cite[Lemma~1.6.11, p.~27]{MR2216774} and noting that the degrees of $F_i$ in each of the $n$ variables sum to at most $n\deg(F_i)$.


For a positive integer $m$, we have
\[\log^+|m|_v:=\log\max\{|m|_v, 1\}=\begin{cases}
\log m & \text{if $v$ is archimedean,}\\
0 & \text{otherwise},
\end{cases}
\]
and so we consolidate the above inequalities into
\[\log\|F_1+\cdots +F_m\|_v\leq \log\|F_1, ..., F_m\|_v+\log^+|m|_v\]
and
\begin{multline*}
\sum_{i=1}^m (\log\|F_i\|_v-n\deg(F_i)\log^+|2|_v)\leq \log\left\|\prod_{i=1}^m F_i\right\|_v \\ \leq \sum_{i=1}^m (\log\|F_i\|_v+n\deg(F_i)\log^+|2|_v).	
\end{multline*}
(where $n$ is the number of variables).

Now let $f(z)=f_1(z)/f_0(z)$, where $f_1(z), f_0(z)\in K[z]$ have no common factor.
Let $F(z, X, Y)\in K[z, X, Y]$ be a homogeneous form in $X$ and $Y$ of degree $D$, say (suppressing the dependence on $z$ for brevity)
\[F(X, Y)=\sum_{i=0}^Da_iX^{D-i}Y^i\]
with $a_i\in K[z]$.
Then from the properties above, for any $i$
\begin{align*}
\log\|a_if_1^{D-i}f_0^i\|_v&\leq \log\|a_i\|_v+(D-i)\log\|f_1\|_v+i\log\|f_0\|_v\\
&\quad+(\deg(a_i)+(D-i)\deg(f_1)+i\deg(f_0))\log^+|2|\\
&\leq  \log\|F\|_v+D\log\|f_1, f_0\|_v \\&\quad+ (\deg_z(F)+D\deg(f))\log^+|2|_v
\end{align*}
and hence
\begin{multline}\label{eq:generalform}
\log\|F(f_1, f_0)\|_v\leq\log\|F\|_v+D\log\|f_1, f_0\|_v \\+ (\deg_z(F)+D\deg(f))\log^+|2|_v+\log^+|D+1|_v.
\end{multline}

Now, if
\[R(z, w)=\frac{b_0+\cdots +b_d w^d}{b_{d+1}+\cdots +b_{2d+1} w^d},\]
with $b_i\in K(z)$ without common factors,
 $A_i(X, Y)$ from Lemma~\ref{lem:res} is a homogeneous form of degree $d-1$ in $X$ and $Y$, each coefficient of which is the determinant of a $(2d-1)\times (2d-1)$ matrix whose entries are among the $b_j$. It follows that
\[\log\| A_i\|_v\leq (2d-1)\log\|R\|_v + \log^+|(2d-1)!|_v + (2d-1)\deg_z(R) \log^+|2|_v\]
while the degree of each coefficient of $A_i$ in $z$ is at most $(2d-1)\deg_z(R)$. It follows from ~\eqref{eq:generalform} that
\begin{multline*}
\log\| A_i(f_1, f_0)\|_v \leq (d-1)\log\| f_1, f_0\|_v+(2d-1)\log\|R\|_v + \log^+|(2d-1)!|_v \\ + (2d-1)\deg_z(R)\log^+|2|_v+(d-1)\deg(f)\log^+|2|_v
\end{multline*}
and similarly for $B_i$.
 
Since $A_i(f_1, f_0)$, $B_i(f_1, f_0)$ are polynomials in $z$ of degree at most $(d-1)\deg(f)$, and $P(f_1, f_0)$ and $Q(f_1, f_0)$ of degree at most $d\deg(f)$, we deduce for $i=1, 0$
\begin{align*}
 \log\| \Res(P, Q)f_i^{2d-1}\|_v & = 
\log\| A_i(f_1, f_0)P(f_1, f_0)+B_i(f_1, f_0)Q(f_1, f_0)\|_v\\
&\leq \log\| A_i(f_1, f_0)P(f_1, f_0), B_i(f_1, f_0)Q(f_1, f_0)\|_v\\&\quad+\log^+|2|_v \\
&\leq \log\| P(f_1, f_0), Q(f_1, f_0) \|_v \\&\quad+ \log\|  A_i(f_1, f_0), B_i(f_1, f_0) \|_v  \\ &\quad+(2d-1)\deg(f)\log^+|2|+\log^+|2|\\
&\leq  \log\| P(f_1, f_0), Q(f_1, f_0) \|_v +(d-1)\log\| f_1, f_0\|_v\\&\quad +(2d-1)\log\|R\|_v + \log^+|(2d-1)!|_v \\ &\quad + \Big((2d-1)\deg_z(R)_v+(3d-2)\deg(f)+1\Big)\log^+|2|
\end{align*}

On the other hand, $\Res(P, Q)\in K[z]$ has degree at most $2d\deg_z(R)$, so
\begin{multline*}
\log\| \Res(P, Q)f_i^{2d-1}\|_v\geq \log\| \Res(P, Q)\|_v+(2d-1)\log\| f_i\|_v \\-2d\deg_z(R)\log^+|2|_v-(2d-1)\deg(f)\log^+|2|_v.
\end{multline*}

Combining these, we have
\begin{multline}\label{eq:nullcrude}
d\log\| f_1, f_0\|_v\leq 	\log\| P(f_1, f_0), Q(f_1, f_0) \|  -\log\| \Res(P, Q)\|_v\\ 
+(2d-1)\log\|R\|_v + \log^+|(2d-1)!|_v  + (4d-1)\deg_z(R)\log^+|2|_v\\+(5d-3)\deg(f)\log^+|2|_v+\log^+|2|.
\end{multline}

At this point we note that $h(f)$ can be computed as a weighted sum of the terms $\log\|f_0, f_1\|_v$, for $v\in M_K$, as in~\eqref{eq:height}. One would like to compute $h(R(z, f(z)))$ by summing the terms $\log\|P(f_1, f_0), Q(f_1, f_0)\|_v$, but this works only if $P(f_1, f_0)$ and  $Q(f_1, f_0)$ have no common factor in $K[z]$, and they very possibly do.

Write $P(f_1, f_0)=g_1r$ and $Q(f_1, f_0)=g_0 r$, where $g_1, g_0\in K[z]$ have no common factor. Since $f_1$ and $f_0$ have no common factor, we have from~\eqref{eq:res} that $r$ divides $\Res(P, Q)$. Note that for $i=1, 0$
\begin{align*}
\log \| g_i r\|_v &\leq \log \| g_i \|_v+\log \| r\|_v+\deg(g_i)\log^+|2|_v+ \deg(r)\log^+|2|_v\\
&\leq  	\log \| g_i \|_v+\log \| r\|_v+(d\deg(f)+\deg_z(R))\log^+|2|_v.
\end{align*}
Writing $\Res(P, Q)=rs$, we have
\[\log\|r\|_v-\log\|\Res(P, Q)\|_v\leq \log\|r\|_v-\log\|rs\|_v \leq -\log\|s\|_v+\deg(\Res(P, Q))\log^+|2|_v,\]
and so from~\eqref{eq:nullcrude} we have
\begin{align}
d\log\| f_1, f_0\|_v &\leq 	\log\|g_0r, g_1r\|_v-\log\| \Res(P, Q)\|_v\nonumber \\ \nonumber
&\quad +(2d-1)\log\|R\|_v + \log^+|(2d-1)!|_v  + (4d-1)\deg_z(R)\log^+|2|_v\\\nonumber&\quad+(5d-3)\deg(f)\log^+|2|_v+\log^+|2|\\\nonumber
&\leq \log\|g_0, g_1\|_v+\log\|r\|_v-\log\| \Res(P, Q)\|_v \\\nonumber
&\quad +(2d-1)\log\|R\|_v + \log^+|(2d-1)!|_v  + 4d\deg_z(R)\log^+|2|_v\\\nonumber&\quad+(6d-3)\deg(f)\log^+|2|_v+\log^+|2|\\\nonumber
&\leq \log\|g_0, g_1\|_v-\log\|s\|_v+\deg(\Res(P, Q))\log^+|2|_v\\\label{eq:local}
&\quad +(2d-1)\log\|R\|_v + \log^+|(2d-1)!|_v  + 4d\deg_z(R)\log^+|2|_v\\&\quad+(6d-3)\deg(f)\log^+|2|_v+\log^+|2|\nonumber
\end{align}
Since $R(z, f(z))=g_1(z)/g_0(z)$ in lowest terms, we have from the definition~\eqref{eq:height} that
\[h(R(z, f(z)))=\sum_{v\in M_K}\frac{[K_v:\QQ_v]}{[K:\QQ]}\log\|g_1, g_0\|_v,\]
just as
\[h(f)=\sum_{v\in M_K}\frac{[K_v:\QQ_v]}{[K:\QQ]}\log\|f_1, f_0\|_v.\]
For the polynomial $s\neq 0$, we set
\[h_{\operatorname{poly}}(s)=\sum_{v\in M_K}\frac{[K_v:\QQ_v]}{[K:\QQ]}\log\|s\|_v,\]
which is the \emph{projective} height of the tuple of coefficients of $s$, not the height of $s$ as a rational function. Note that
\[0\leq h_{\operatorname{poly}}(s)\leq h(s),\]
where the second inequality follows from $\log \leq \log^+$, and the first from the product formula and $|s_i|_v\leq \|s\|_v$ for any $i$.

Note also that, for any integer $m$,
\[\sum_{v\in M_K}\frac{[K_v:\QQ_v]}{[K:\QQ]}\log^+|m|_v=\log m,\]
and so summing~\eqref{eq:local}
 over all $v\in M_K$, weighting by the local degree $\frac{[K_v:\QQ_v]}{[K:\QQ]}$, we obtain
\begin{multline}\label{eq:heightbound}
dh(f)\leq dh(f)+h_{\operatorname{poly}}(s)\leq h(R(f))+(2d-1)h(R)+\Big(\deg(\Res(P, Q))\\+4d\deg_z(R)+(6d-3)\deg(f)+1\Big)\log 2+\log (2d-1)!
\end{multline}
as $h_{\operatorname{poly}}(s)\geq 0$.
 \end{proof}
 
 \begin{example}
Our interest in Lemma~\ref{lem:height} lies largely in the existence of an explicit bound, but note that in special cases one can often do far better than~\eqref{eq:heightbound} by carrying out the elimination of variables directly. For example, let
\[R(z, w)=\frac{w^d+z}{w},\]
so that $P=w^d+z$, $Q=w$, and $\Res(P, Q)=(-1)^dz$, which has degree one and height zero. By~\eqref{eq:heightbound} we have
\begin{equation}\label{eq:badbound}
dh(f)\leq h\left(\frac{f^d+z}{f}\right)+((6d-3)\deg(f)+4d+2)\log 2 + \log (2d-1)!,
	\end{equation}
	and we remind the reader that $\log n!\approx n\log n$.
By a more direct calculation, though, we have
\begin{align*}
zX^{2d-1}&=zX^{d-1}(X^d+zY^d)-z^2X^{d-2}Y(XY^{d-1})\\
 zY^{2d-1}&=Y^{d-1}(X^d+zY^d)-X^{d-1}(XY^{d-1}).
\end{align*}
If follows that for any absolute value $v$ on $K$, we have (since $\|\pm z g(z)\|_v=\|g\|_v$)
\begin{align*}
(2d-1)\log\|f_1, f_0\|_v	& \leq \log\|f_1^d+zf_0^d, f_1f_0^{d-1}\|_v+\log\|zf_1^{d-1}, z^2f_1^{d-2}f_0, f_0^{d-1}, f_1^{d-1}\|_v\\&\quad +\log^+|2|_v\\
&\leq \log\|f_1^d+zf_0^d, f_1f_0^{d-1}\|_v+(d-1)\log\|f_1, f_0\|_v\\&\quad +((d-1)\deg(f)+1)\log^+|2|_v,
\end{align*}
whence
\begin{equation}\label{eq:localimproved}d\log\|f_1, f_0\|_v\leq \log\|f_1^d+zf_0^d, f_1f_0^{d-1}\|_v+((d-1)\deg(f)+1)\log^+|2|_v.\end{equation}
By the observation $\Res(P, Q)=\pm z$, the greatest common factor of $f_1^d+zf_0^d$ and $f_1f_0^{d-1}$ is either 1 or $z$, and again we have $\|\pm zg(z)\|_v=\|g\|_v$ for any polynomial $g$, and hence summing~\eqref{eq:localimproved} over all places, we obtain
\[dh(f)\leq h\left(\frac{f^d+z}{f}\right)+((d-1)\deg(f)+1)\log 2,\]
a clear improvement on~\eqref{eq:badbound} for computational purposes.
\end{example}

Our next lemma estimates the effect of substitutions, or taking derivatives, on the height of a rational function. 

\begin{lemma}\label{lem:subs}
Let $f(z)\in \overline{\QQ}(z)$. Then
\begin{align*}h(f(z+1))&\leq h(f)+\deg(f)\log 2 + \log (\deg(f)+1)\\
\intertext{and}
h(f')&\leq 2h(f)+4\deg(f)\log 2.
\end{align*}
\end{lemma}

\begin{proof}
We have for any $g(z)=\sum b_iz^i\in K[z]$,
\begin{align*}
\log\| g(z+1)\|_v &=\log\left\| \sum_{i=0}^{\deg(g)} b_i(z+1)^i\right\|_v\\
&=\log\left\| \sum_{i=0}^{\deg(g)} b_i\sum_{j=0}^i\binom{i}{j}z^j\right\|_v\\
&=\log\left\| \sum_{j=0}^{\deg(g)}\sum_{i=j}^{\deg(g)} b_j\binom{i}{j}z^j\right\|_v\\
&=\log\max_{0\leq j\leq \deg(g)}\left|\sum_{i=j}^{\deg(g)} b_j\binom{i}{j}\right|_v\\
&\leq  \log\max|b_j|_v+\deg(g)\log^+|2|_v+\log^+|\deg(g)+1|_v\\
& = \log\| g(z)\|_v +\deg(g)\log^+|2|_v+\log^+|\deg(g)+1|_v,
\end{align*}
since $\binom{i}{j}\leq 2^i\leq 2^{\deg(g)}$,
and so
\[\log\| f_1(z+1), f_0(z+1)\|_v\leq \log\| f_1(z), f_0(z)\|_v+\deg(f)\log^+|2|_v+\log^+|\deg(f)+1|_v.\]
Again, proceeding with the weighted sum~\eqref{eq:height}, we have
\[h(f(z+1))\leq h(f)+\deg(f)\log 2 + \log (\deg(f)+1),\]
where we note that $f_1(z+1)$ and $f_0(z+1)$ cannot have a common factor unless $f_1(z)$ and $f_0(z)$ did.

Similarly, to estimate $h(f)$ we will provide an upper bound on the quantity $\log\|f_1'f_0-f_0'f_1, f_0^2\|_v$ in each absolute value. First, note from ... above that
\[\log\|f_0^2\|_v\leq 2\log\|f_0\|_v+2\deg(f)\log^+|2|_v\leq 2\log\|f_1, f_0\|_v+2\deg(f)\log^+|2|_v.\]
On the other hand, for any polynomial $g(z)=\sum b_iz^i\in K[z]$
\begin{align*}
\log\|g'(z)\|_v	&=\log\left\|\sum_{i=1}^{\deg(g)}ib_{i}z^{i-1}\right\|_v\\
&\leq  \log\max_{1\leq i\leq \deg(g)}|b_i|_v+\log^+|\deg(g)|_v\\
&\leq  \log\|g\|_v+\log^+|\deg(g)|_v.
\end{align*}

\begin{align*}
\log\|f_1'f_0-f_0'f_1\|_v&\leq \log\max_{i=0, 1}\|f_i'f_{1-i}\|_v +\log^+|2|_v\\
&\leq  \max_{i=0, 1}\big(\log\|f_i'\|_v+\log\|f_{1-i}\|_v+\deg(f_i')\log^+|2|_v\\&\quad+\deg(f_{1-i})\log^+|2|_v\big)+\log^+|2|_v\\
&\leq  2\log\|f_1, f_0\|_v+2\deg(f)\log^+|2|_v.
\end{align*}
In any case, we have
\begin{equation}\label{eq:wronk}\log\|f_1'f_0-f_0'f_1, f_0^2\|_v\leq 2\log\|f_1, f_0\|_v+2\deg(f)\log^+|2|_v.\end{equation}
Now, it is of course possible that $f_1'f_0-f_0'f_1$ and $f_0^2$ will have a common factor. Write
$f_1'f_0-f_0'f_1=g_1r$ and $f_0^2=g_0r$, so that
\begin{align*}
\log\|g_i\|_v&\leq \log\|g_ir\|_v-\log\|r\|_v+\deg(g_i)\log^+|2|_v+\deg(r)\log^+|2|_v\\	
&\leq  \log\|g_ir\|_v-\log\|r\|_v+2\deg(f)\log^+|2|_v.
\end{align*}
Combining this with~\eqref{eq:wronk}, we have by summing over all places that
\[h(f'(z))\leq 2h(f)-h_{\operatorname{poly}}(r)+4\deg(f)\log 2\]
proving the claim, since $h_{\operatorname{poly}}(r)\geq 0$. 
\end{proof}

Before continuing, we note that the solutions to~\eqref{eq:diff} and~\eqref{eq:eq} are now rather constrained if all have rational coefficients. If $R(z, w)\in \QQ(w, z)$ has $d=\deg_w(R)\geq 2$, and $f(z)\in\QQ(z)$ is a solution to~\eqref{eq:eq}, then it follows from Lemmas~\ref{lem:degree},  \ref{lem:height}, and~\ref{lem:subs} (and the fact that $\deg(\Res(P, Q))\leq 2d\deg_z(R)$) that
\begin{align*}
	dh(f)&\leq h(R(f))+(2d-1)h(R)+((24d-9)\deg_z(R)+1)\log 2 \\&\quad + \log(2d-1)!\\
	&= h(f(z+1))+(2d-1)h(R)+((24d-9)\deg_z(R)+1)\log 2 \\&\quad + \log(2d-1)!\\
	&\leq h(f)+(2d-1)h(R)+((24d-6)\deg_z(R)+1)\log 2 \\&\quad + \log(2d-1)!+\log(3\deg_z(R)+1),
\end{align*}
whence
\begin{multline*}
h(f)\leq \frac{1}{d-1}\Big((2d-1)h(R)+((24d-6)\deg_z(R)+1)\log 2 \\ + \log(2d-1)!+\log(3\deg_z(R)+1)\Big).
\end{multline*}
In particular, the degree of $f$ is bounded, and the coefficients of $f$ are drawn from a finite set (depending on $R$). It is not yet clear, though, that there are finitely many solutions to~\eqref{eq:eq} given just that $R$ has rational coefficients, since \emph{a priori} the solutions might not.

\begin{remark}\label{rm:gendiff}
We note also that there is nothing particularly special in Lemma~\ref{lem:subs} about the substitution $z\mapsto z+1$. In general, if $\sigma\in\operatorname{Aut}(\overline{\QQ}(z)/\overline{\QQ})$	then $\sigma(z)$ is a fractional linear transform with algebraic coefficients, and the triangle inequality (as in the proof of Lemma~\ref{lem:subs}) gives
\[h(f\circ\sigma)\leq h(f)+dh(\sigma)+2\deg(f)\log 2 +\log(\deg(f)+1).\]
It follows that the machinery in this paper could easily be used to study generalized difference equations of the form
\begin{equation}\label{eq:gendiff}f\circ \sigma(z)=R(z, f(z)),\end{equation}
with $\sigma\in\operatorname{Aut}(\PP^1)$ fixed. Indeed, the argument in the previous paragraph shows that the set of $f(z)\in\overline{\QQ}(z)$ solving \emph{any} generalized difference equation of the form~\eqref{eq:gendiff} with $\sigma$ and $R$  drawn from sets of bounded height and degree (in the case of $R$) is itself a set of bounded height. It is not clear to the author, however, whether one should expect to be able to extend Yanagihara's result to this setting.
\end{remark}

\begin{lemma}
Let $R(z, w)\in \CC(z, w)$ with $d=\deg_w(R)\geq 2$, and fix $k$. Then the set $X$ of $f\in \operatorname{Hom}_k(\CC)$ solving either~\eqref{eq:diff} or~\eqref{eq:eq} is Zariski closed. Furthermore, the degrees of the irreducible components of $X$ sum to at most $(d+1)^{(d+1)k+\deg_z(R)}$, in the case of solutions to~\eqref{eq:eq}, or $(d+2)^{(d+2)k+\deg_z(R)}$ in the case of solutions to~\eqref{eq:diff}.
\end{lemma}

\begin{proof}
Write a generic $f$ of degree $k$ as in~\eqref{eq:ratd}.
Then 
\[f_1(z+1)Q(f_1(z), f_0(z))-f_0(z+1)P(f_1(z), f_0(z))=\sum_{i=0}^{(d+1)k+\deg_z(R)}\Phi_i(\mathbf{c})z^i\]
for some homogeneous  forms $\Phi_i$ of degree $d+1$ in $\mathbf{c}$, which depend on $P$ and $Q$. A solution to~\eqref{eq:eq} is given exactly by the simultaneous vanishing of these homogeneous forms (some of which might already be the zero form). It follows from \cite[Theorem~7.7, p.~53]{MR0463157} that if $Z\subseteq \PP^{2k+1}$ is Zariski closed and irreducible, and $H\subseteq \PP^{2k+1}$ is a hypersurface, then either $Z\subseteq H$ or the irreducible components of $H\cap Z$ have degree summing to at most $\deg(H)\deg(Z)$. By induction, the intersection of $(d+1)k+\deg_z(R)$ homogeneous forms of degree $d+1$ has irreducible components of degree summing to at most $(d+1)^{(d+1)k+\deg_z(R)}$. This proves the statement for the intersection of the hypersurfaces $\Phi_i=0$ in $\PP^{2k+1}$, and the irreducible components of the intersection in $\operatorname{Hom}_k$ are just the intersections with $\operatorname{Hom}_k$ of those components on which $\Res(\mathbf{c})$ does not vanish identically.

In the case of solutions to~\eqref{eq:diff}, we have
\begin{multline*}
(f_1'(z)f_0(z)-f_0'(z)f_1(z))Q(f_1(z), f_0(z))-f_0^2(z)P(f_1(z), f_0(z))\\=\sum_{i=0}^{(d+2)k+\deg_z(R)}\Phi_i(\mathbf{c})z^i,	
\end{multline*}
where the $\Phi_i$ are homogeneous forms of degree $d+2$. The same argument as above now gives that the vanishing of these forms (which corresponds exactly with solutions to~\eqref{eq:diff}), defines a Zariski-closed subset of $\operatorname{Hom}_k$ of degree at most $(d+2)^{(d+2)k+\deg_z(R)}$.
\end{proof}

\section{The proof of the main result}

\begin{proof}[Proof of Theorem~\ref{th:finite}]
Let $R(z, w)\in \CC(z, w)$, and let $f_1(z), ..., f_\ell(z)\in \CC(z)$ be some distinct solutions to~\eqref{eq:eq}, all of degree exactly $k$. Let $F\subseteq \CC$ be the subfield generated over $\QQ$ by the coefficients of $R$ and the $f_i$, a finite set of complex numbers. We will prove by induction on the transcendence degree of $F$ over $\QQ$ that \[\ell\leq (d+1)^{(d+1)k+\deg_z(R)}.\] Since $k\leq 3\deg_z(R)$, by Lemma~\ref{lem:degree}, the total number of solutions to~\eqref{eq:eq} in $\CC(z)$ will be at most the estimate in~\eqref{eq:B}.

Consider first the base case, where $R$ and the $f_i$ are defined over the algebraic numbers. Then the $f_i$ correspond to points in $X\subseteq \operatorname{Hom}_{k}$, the irreducible components of which have degree summing to at most $(d+1)^{(d+1)k+\deg_z(R)}$. It suffices to show that these irreducible components are points, so suppose to the contrary that $X$ has a component of positive dimension. Then $X$ contains a curve $Y$, which admits a non-constant map to $\PP^1$, all defined over some number field, and so there is some $D\geq 0$ such that $Y(\overline{\QQ})$ contains infinitely many points of algebraic degree at most $D$. But by Lemma~\ref{lem:height}, $Y(\overline{\QQ})$ is a set of bounded height, contradicting Lemma~\ref{lem:northcott}. This completes the proof for the case in which $F\subseteq \overline{\QQ}$.

Now suppose  that $F$ contains transcendental elements, and that the inequality $\ell\leq (d+1)^{(d+1)k+\deg_z(R)}$ is know for all fields of lower transcendence degree.
Let $K\subseteq F$ be a maximal subfield of transcendence degree one less than that of $F$. Then $F$ is isomorphic over $K$ to the function field $K(C)$ of some algebraic curve $C/K$. We will fix coefficients of $R$ and $f_i$ in $F$, and identify them with their images in $K(C)$, which are rational functions on $C$ defined over $k$. Since there are only finitely many of these functions, there is an affine open $U\subseteq C$ on which all are regular. For $P\in C(\overline{K})$, we may evaluate these coefficients at $P$ to obtain a rational function $R_{P}(z, w)\in \overline{K}(z, w)$ and $f_{i, P}(z)\in \overline{K}(z)$. Note that the degrees of the specializations may be less than the original degrees.

Now, once we have written each $f_i$ as in~\eqref{eq:ratd}, we may compute the resultant of the numerator and denominator of $f_i$ with the chosen coefficients, and obtain a not-identically-zero regular function $\Res(f_i)\in K[U]$, with the property that $\Res(f_i)(P)=0$ if and only if $\deg(f_{i, P})<k$. We also have $\Res(P, Q)\in K[U, z]$, and we will choose a not-identically-zero coefficient $r$ of this polynomial, and a coefficient $s$ of the largest power of $z$ appearing in $R$. There is a Zariski open $V\subseteq U$ on which $r$, $s$, and the $\Res(f_i)$ as nowhere vanishing, and for $P\in V(\overline{K})$, the specializations $f_{1, P}, ..., f_{\ell, P}\in \overline{K}(z)$ are solutions of degree exactly $k$ to $g(z+1)=R_{P}(z, g(z))$, and $\deg_w(R_P)=\deg_w(R)$ and $\deg_z(R_P)=\deg_z(R)$.

 Finally, since the $f_{i}$ are distinct, there is for every $i\neq j$ some cross-ratio of coefficients $c_{i, s}c_{j, t}-c_{i, t}c_{j, s}\in k[V]$ which is not identically zero. For each pair $i\neq j$ we choose such a cross-ratio, and an affine $W\subseteq V$ on which none of these functions vanish. So for any $P\in W(\overline{k})$ the specializations $f_{1, P}, ..., f_{\ell, P}$ are \emph{distinct} solutions of degree $k$ to the specialized difference equation, all defined over the field $\overline{K}$ of transcendence degree one less than that of $F$. The induction hypothesis applies to this example, completing the proof that $\ell\leq (d+1)^{(d+1)k+\deg_z(R)}$ in general, and hence of Theorem~\ref{th:finite}(A). The proof of Theorem~\ref{th:finite}(B) is analogous.
\end{proof}


Finally, we justify our assertion in the introduction that solutions to~\eqref{eq:diff} and~\eqref{eq:eq} must have algebraic coefficients when $R$ does (subject to the usual hypotheses on $\deg_w(R)$).

\begin{corollary}
Let $R(z, w)\in K(z, w)$ for some subfield $K\subseteq \CC$. Then every solution to~\eqref{eq:diff} and \eqref{eq:eq}	(with $\deg_w(R)\geq 3$ or $\deg_w(R)\geq 2$ respectively) in $\CC(z)$ is already contained in $\overline{K}(z)$. In particular, if $R$ has algebraic coefficients, then so does any solution $f$.
\end{corollary}

\begin{proof}
Any solution $f(z)\in \CC(z)$ is defined over some finitely-generated extension of $K$, which is isomorphic to $K(Z)$ for some irreducible algebraic variety $Z$. 	If $f$ has degree $k$, then $f$ induces a map $Z\to \operatorname{Hom}_k$ defined over $\CC$, whose image is not contained in the resultant locus. But since the solutions to~\eqref{eq:eq} or~\eqref{eq:diff} in $\operatorname{Hom}_k$ are a finite union of zero-dimensional subvarieties, it follows that this map is constant, and hence $f$ was already defined over $K$.
\end{proof}

\bibliography{nevan}
\bibliographystyle{plain}

\end{document}